\documentclass{amsart}
\usepackage{amsmath}
\usepackage{amssymb}
\usepackage{amsthm}
\usepackage{bm}

\newtheorem{theorem}{Theorem}
\newtheorem{prop}[theorem]{Proposition}
\newtheorem{lemma}[theorem]{Lemma}
\newtheorem{rem}[theorem]{Remark}
\newtheorem{exmp}[theorem]{Example}
\newtheorem{cor}[theorem]{Corollary}

\begin{document}

\title{On maximal cliques in the graph of simplex codes}
\author{Mariusz Kwiatkowski*}\thanks{* corresponding author}
\author{Mark Pankov}
\keywords{Grassmann graph, simplex code, equidistant code}
\subjclass[2020]{51E20,51E22}
\address{Faculty of Mathematics and Computer Science, 
University of Warmia and Mazury, S{\l}oneczna 54, 10-710 Olsztyn, Poland}
\email{mkw@matman.uwm.edu.pl, pankov@matman.uwm.edu.pl}

\begin{abstract}
The induced subgraph of the corresponding Grassmann graph formed by simplex codes is considered.
We show that this graph, as the Grassmann graph, contains two types of maximal cliques. 
For any two cliques of the first type there is a monomial linear automorphism transferring one of them to the other.
Cliques of the second type are more complicated and can contain different numbers of elements.
\end{abstract}

\maketitle

\section{Introduction}
The Grassmann graph of indices $n,k,q$
is the simple graph whose vertices are $k$-dimensional subspaces of an $n$-dimensional vector space over 
the $q$-element field and two such subspaces are connected by an edge if their intersection is $(k-1)$-dimensional.
It is natural to assume that $1<k<n-1$ (otherwise, we obtain a complete graph).
This graph is identified with the graph whose vertices are linear $[n, k]_q$ codes and two such codes are connected by an edge if they have the maximal number of common codewords.
The induced subgraphs of Grassmann graphs formed by non-degenerate linear $[n, k]_q$ codes are investigated in \cite{KP1,KP2,Pank-code1,Pank-code2},
the induced subgraphs related to projective codes and codes with lower bounded minimal dual distance are studied in \cite{CGK,KPP}.
Note that some structural properties of these graphs (path distance, maximal cliques) are essentially dependent on the parameters $n,k,q$.

We direct our attention to the induced subgraphs of Grassmann graphs  formed by simplex codes.
A $q$-ary simplex code of dimension $k$ is a non-degenerate linear code of length $n=\frac{q^k-1}{q-1}$
and the Hamming distance between any two distinct code words is $q^{k-1}$; all such codes are equivalent by the MacWilliams theorem.
Simplex codes are dual to Hamming codes and are first-order Reed–Muller codes \cite[Subsection 1.2.2]{TVN}.
Also, every non-degenerate equidistant code is equivalent to some replication of a simplex code \cite{B}. 

The induced subgraph of the $(n,k,q)$-Grassmann graph ($n=\frac{q^k-1}{q-1}$) consisting of 
$q$-ary simplex codes of dimension $k$ will be denoted by $\Gamma^s(k,q)$.
There is only one binary simplex code of dimension $2$ and the corresponding graph is a single vertex.
The graph $\Gamma^s(3,2)$ is isomorphic to the incidence graph formed by $1$-dimensional and $3$-dimensional subspaces
of a $4$-dimensional vector space over the two element field \cite[Corollary 2]{KPP} and 
$\Gamma^s(2,3)$ is isomorphic to the complete bipartite graph $K_{4,4}$ \cite[Example 2]{KP-simp2}.
The graph $\Gamma^s(2,4)$ is more complicated; 
its metric properties and the action of the group of monomial semilinear automorphisms on this graph are described in \cite{KP-simp2}.

In this paper, we investigate maximal cliques of $\Gamma^s(k,q)$.
A maximal clique of the Grassmann graph is precisely of one of the following types: 
a star (formed by all $k$-dimensional subspaces containing a certain $(k-1)$-dimensional subspace) or 
a top (consisting of all $k$-dimensional subspaces contained in a certain $(k+1)$-dimensional subspace).
The intersection of a star or a top with the set of simplex codes is a clique of $\Gamma^s(k,q)$ (if it is non-empty);
however, this clique need not be maximal. 
We will say that this intersection is a star or a top in $\Gamma^s(k,q)$ only in the case when it is a maximal clique. 
It is easy to see that 
every star of $\Gamma^s(k,q)$ corresponds to a $(k-1)$-dimensional subcode of a simplex code and
for any two such stars there is a monomial linear automorphism transferring one of them to the other.
Every maximal clique in $\Gamma^s(3,2)$ and $\Gamma^s(2,3)$ is a star and a top simultaneously 
(since these graphs are bipartite) and all maximal cliques of $\Gamma^s(2,4)$ are stars. 
Our main result (Theorem \ref{theorem-tops}) states that tops (which are not stars) exist in each of the remaining cases;
furthermore, if $k\ge 3$ and $q\ge 3$, then there are tops of $\Gamma^s(k,q)$ containing different numbers of elements.
In the case when $k\ge 5$ and $q\ge 3$, we construct a top of $\Gamma^s(k,q)$ containing  precisely three elements.

At the end, we consider the point-line geometry whose maximal singular subspaces are $q$-ary simplex codes of dimension $k$. 
This is a polar space only in the case when $k=2$, $q=3$ or $k=3$, $q=2$.
As an application of the main result, we show that in almost all remaining cases (these cases are determined in Theorem \ref{theorem-tops})
 lines of this geometry cannot be characterized in terms of the collinearity relation as lines of polar spaces.

\section{Simplex codes}
Let ${\mathbb F}=\mathbb{F}_{q}$ be the field consisting of $q$ elements.
Consider the $n$-dimensional vector space $V={\mathbb F}^{n}$
over this field. The standard basis of $V$ is formed by the vectors 
$$\bm{e}_{1}=(1,0,\dots,0),\dots,\bm{e}_{n}=(0,\dots,0,1).$$
Denote by $C_{i}$ the kernel of the $i$-th coordinate functional $(x_{1},\dots,x_{n})\to x_{i}$,
i.e. the hyperplane of $V$ spanned by all $\bm{e}_j$ with $j\ne i$.
Recall that an $m$-dimensional vector space over ${\mathbb F}_q$ contains precisely 
$$[m]_q=\frac{q^{m}-1}{q-1}=q^{m-1}+\dots+q+1$$
distinct $1$-dimensional subspaces and the same number of hyperplanes.

Every $k$-dimensional subspace of $V$ is interpreted as a {\it linear $[n,k]_{q}$ code}
and their vectors as {\it code words} of this code.
Two such codes are {\it equivalent} if there is a monomial linear automorphism of $V$ sending one of these codes to the other
(a linear automorphism of $V$ is monomial if it maps each $\bm{e}_i$ to a scalar multiple of a certain $\bm{e}_j$). 

We will consider linear codes only.
A code is said to be $t$-{\it equidistant} if the Hamming distance between any two distinct code words in this code 
is $t$ or, equivalently, the Hamming weight of each non-zero code word in the code is $t$
(the Hamming distance between $\bm{x}$ and $\bm{y}$ is equal to the Hamming weight of $\bm{x}-\bm{y}$).
By MacWillams theorem, every linear isomorphism between two codes of $V$
preserving the Hamming weight can be uniquely extended to a monomial linear automorphism of $V$.
Therefore, if $C,C'\subset V$ are $t$-equidistant codes of the same dimension, then
every linear isomorphism of $C$ to $C'$ is extendible to a monomial linear automorphism of $V$, in particular,
these codes are equivalent.

A  $k$-dimensional  code $C\subset V$ is {\it non-degenerate} if the restriction of every coordinate functional to $C$ is non-zero,
in other words, there is no hyperplane $C_i$ containing $C$.
Suppose that $\bm{x}_{1},\dots,\bm{x}_{k}$ is a basis of $C$ and $\bm{x}^{*}_{1},\dots,\bm{x}^{*}_{k}$ is the dual basis of $C^{*}$
satisfying $\bm{x}^{*}_{i}(\bm{x}_{j})=\delta_{ij}$ (Kronecker delta).
If $M$ is the generator matrix of $C$ whose rows are $\bm{x}_{1},\dots,\bm{x}_{k}$ and
$(a_{1j},\dots, a_{kj})^t$ is the $j$-column of $M$, then the restriction of the $j$-th coordinate functional to $C$ is
$$a_{1j}\bm{x}^{*}_{1}+\dots+a_{kj}\bm{x}^{*}_{k}.$$
Therefore, $C$ is non-degenerate if and only if the generator matrices of  $C$ do not contain zero columns. 

From this moment, we assume that $n=[k]_q$ and $k\ge2$.
Let $C$ be a non-degenerate $k$-dimensional code of $V$ such that
$$C\cap C_i\ne C\cap C_j\;\mbox{ if }\;i\ne j,$$
i.e. the restrictions of the coordinate functionals to $C$ are mutually non-proportional.
This means that columns in the generator matrices of $C$ are non-zero and  mutually non-proportional,
in other words,  there is a one-to-one correspondence between columns in a generator matrix of $C$  and $1$-dimension subspaces of ${\mathbb F}^k$.
In this case, $C$ is called a {\it $q$-ary simplex codes of dimension $k$}.
Every hyperplane of $C$ is a certain $C\cap C_{i}$ (since $C$ contains precisely $n=[k]_q$ hyperplanes)
and every $m$-dimensional subspace of $C$ is the intersection of precisely $[k-m]_q$ distinct $C\cap C_i$.
In particular, every non-zero code word of $C$ has precisely $[k-1]_q$ coordinates equal to $0$ and 
its Hamming weight  is 
$$[k]_q-[k-1]_q=q^{k-1}.$$
So, $q$-ary simplex codes of dimension $k$ are $q^{k-1}$-equidistant and, consequently, all such codes are equivalent. 
It follows from \cite{B} (see also \cite[Theorem 7.9.5]{HP} and \cite{Ward}) that $q$-ary simplex codes of dimension $k$ can be characterized as maximal $q^{k-1}$-equidistant codes of $V$.

The group of monomial linear automorphisms of $V$ acts transitively on the set of simplex codes
and contains precisely $n!(q-1)^{n}$ elements. 
Since every linear automorphism of a simplex code can be uniquely extended to a monomial linear automorphism of $V$,
there are precisely $$(q^{k}-1)(q^{k}-q)\dots(q^{k}-q^{k-1})$$ (the number of elements in ${\rm GL}(k,q)$) monomial linear automorphisms of $V$
which preserve a fixed simplex code. Therefore,
\begin{equation}\label{numb1}
\frac{n!(q-1)^{n}}{(q^{k}-1)(q^{k}-q)\dots(q^{k}-q^{k-1})}
\end{equation}
is the number of $q$-ary simplex codes of dimension $k$ (recall that $n=[k]_q$).
In particular, there is a unique binary simplex code of dimension $2$.
For all other cases there is more than one simplex code.

A vector $(x_1,\dots,x_n)\in V$ is a non-zero code word of a simplex code if and only if its Hamming weight is $q^{k-1}$.
By \cite{KPP}, all such vectors form the algebraic  variety defined by the equations
$$\sum_{i_1<\dots< i_{p^{j}}}x_{i_1}^{q-1}\dots\; x_{i_{p^{j}}}^{q-1}  =  0$$
for $j\in\{0,\dots,mk-m-1\}$, where $q=p^m$ and $p$ is a prime number.
This variety is a quadric only in the following two cases: $q=2,\;k=3$ and $q=3,\;k=2$. 

\section{Subcodes of simplex codes}
As above, we assume that $n=[k]_q$ and $k\ge 2$.
A code of $V$ is $q^{k-1}$-equidistant if and only if it is a subcode of a simplex code.
We start from the following characterization of such codes in terms of generator matrices.

\begin{prop}\label{matrix}
If $X$ is an $m$-dimensional subcode of a simplex code, then every generator matrix $M$ of $X$
satisfies the following condition:
\begin{enumerate}
\item[$(*)_m$] $M$ contains precisely $[k-m]_{q}$ zero columns
and any non-zero column of $M$  is proportional to precisely $q^{k-m}$ columns including itself.
\end{enumerate} 
If a generator matrix $M$ of an $m$-dimensional code $X\subset V$ satisfies $(*)_m$, 
then $X$ is a subcode of a simplex code.
\end{prop}

\begin{proof}
Suppose that $X$ is an $m$-dimensional subcode of  a simplex code $C$ and $M$ is a generator matrix of $X$.
Every hyperplane of $C$ is the intersection of $C$ with a certain coordinate hyperplane $C_i$
and $C\cap C_i\ne C\cap C_j$ if $i\ne j$.
Since $X$ is an $m$-dimensional,
the number of hyperplanes of $C$ containing $X$ is $[k-m]_q$ and, consequently,
there are precisely $[k-m]_q$ distinct $i$ such that $X\subset C_i$.
The latter inclusion is equivalent to the fact that 
$i$ columns in the matrix $M$ are zero.
If the $i$-th and $j$-th columns in $M$ are non-zero,
then $X\cap C_i$ and $X\cap C_j$ are hyperplanes of $X$;
furthermore,
the columns  are proportional if and only if these hyperplanes are coincident.
Every hyperplane $H$ of $X$ is contained in precisely $[k-m+1]_{q}$ distinct $C_{i}$.
We remove from this collection the $[k-m]_{q}$ coordinate hyperplanes containing $X$.
The remaining 
$$[k-m+1]_q - [k-m]_q = q^{k-m}$$
coordinate hyperplanes  intersect $X$ precisely in $H$.
The corresponding columns in the matrix $M$ are proportional and there is no any other column in $M$
proportional to them.

Now, we assume that a generator matrix $M$ of an $m$-dimensional code $X\subset V$ satisfies $(*)_m$.
It is sufficient to extend $M$ to a $(k\times n)$-matrix whose columns are non-zero and mutually non-proportional.
We can act on $X$ by the corresponding monomial linear automorphism of $V$ and assume that 
any two proportional columns in $M$ are coincident.
The matrix $M$ contains precisely $[k-m]_q$ zero columns.
We take any $q^{k-m-1}$ such columns and extend them by $1$. 
The remaining 
$$[k-m-1]_q=[k-m]_q - q^{k-m-1}$$
zero columns will be extended by $0$.
Consider any collection of $q^{k-m}$ coincident columns in $M$.
For each element of the field we choose $q^{k-m-1}$ columns from this collection and extend them by this element.
We obtain an $[(m+1)\times n]$-matrix which contains precisely $[k-m-1]_{q}$ zero columns
and any non-zero column in this matrix  is proportional to precisely $q^{k-m-1}$ columns including itself.
The required extension can be contructed recursively.
\end{proof}

\begin{theorem}\label{extension}
Let $m\in \{0,1,\dots,k-1\}$.
Every $m$-dimensional subcode of a simplex code  is contained in precisely 
\begin{equation}\label{numb2}
\frac{[k-m]_q!(q-1)^{[k-m]_q}(q^{k-m}!)^{[m]_q}}{(q^{k-m}-1)(q^{k-m}-q)\cdots(q^{k-m}-q^{k-m-1})q^{m(k-m)}}
\end{equation}
distinct simplex codes.
Furthermore, there are two simplex codes whose intersection is precisely this subcode except the case when 
$q=k=2$.
\end{theorem}

If $q=k=2$, then there is only one simplex code and the second statement fails.
If $m=0$, then \eqref{numb2} is equal to \eqref{numb1} and the second part of Theorem \ref{extension}
states that there is a pair of simplex codes whose intersection is zero. 
Since the group of monomial linear automorphisms of $V$ acts transitively on the set of simplex codes, 
the latter implies that for every simplex code $C\subset V$ there is a simplex code $C'\subset V$ such that $C\cap C'=\bm{0}$.

\begin{proof}[Proof of Theorem \ref{extension}]
Let $S$ be a generator matrix of an $m$-dimensional subcode of a simplex code.
By Proposition \ref{matrix}, $S$ contains precisely $[k-m]_q$ zero columns 
and the set of non-zero columns is the disjoint union of $[m]_q$ sets of size $q^{k-m}$ such that two columns are proportional if and only if 
they belong to the same subsets. 
As in the proof of Proposition \ref{matrix}, we can assume that any two proportional columns in $S$ are coincident.
We extend this matrix to a generator matrix of a simplex code in two steps.

First, we observe that the number of zero-columns in $S$ is equal to the number of all non-zero vectors of ${\mathbb F}^{k-m}$
whose first non-zero coordinates are equal to $1$.
We consider any one-to-one correspondence between zero columns and such vectors and extend every
zero column by a non-zero multiple of the corresponding vector. 
There are precisely $$[k-m]_q!(q-1)^{[k-m]_q}$$ possibilities for such an extension.
Next, for each subset formed by $q^{k-m}$ coincident columns of $S$
we consider any one-to-one correspondence between these columns and vectors of ${\mathbb F}^{k-m}$
and extend every column by the corresponding vector. 
It can be done in $(q^{k-m})!$ distinct ways.
So, we obtain
$$[k-m]_q!(q-1)^{[k-m]_q}\cdot ((q^{k-m})!)^{[m]_q}$$
possible extensions of $S$ to a $(k \times n)$-matrix. 
Any two distinct columns in such an extension are non-proportional, i.e. it is a generator matrix of a simplex code. 
It is clear that any extension of $S$ to a generator matrix of a simplex code can be obtained in the way described above.

Two such extensions $M$ and $M'$ are generator matrices of the same simplex code if and only if 
$M'=TM$, where $T$  is a non-singular $k$-matrix.
Since the first $m$ rows in $M$ and $M'$ are coincident,
$$T=\left[\begin{array}{cc}
I_{m}&\bm{0}\\
A &B\\
\end{array}
\right],$$
where $A$ is an arbitrary $[(k-m)\times m]$-matrix  and $B$ is a non-singular $(k-m)$-matrix.
Since we have $q^{m(k-m)}$ possibilities for $A$ and  
$$(q^{k-m}-1)(q^{k-m}-q)\cdots(q^{k-m}-q^{k-m-1})$$ possibilities for $B$,
there are precisely 
$$(q^{k-m}-1)(q^{k-m}-q)\cdots(q^{k-m}-q^{k-m-1})\cdot q^{m(k-m)}$$
such matrices $T$.
This implies that  \eqref{numb2} is the number of all
simplex codes containing a given $m$-dimensional subcode.

Let $C$ be an $m$-dimensional subcode of a simplex code and $m\in \{0,1,\dots,k-1\}$.
We need to show that there are two simplex codes whose intersection is $C$.
We assume that $q\ne 2$ or $k\ne 2$ (otherwise, there is only one simplex code). 
If
$$q=2,\;k=3\;\mbox{ or }\;q=3,\;k=2,$$ 
then the set of code words of all simplex codes is a quadric whose maximal singular subspaces are precisely simplex codes;
in this case, the statement follows from the fact that every singular subspace is the intersection of two distinct maximal singular subspaces,
in particular, there are maximal singular subspaces whose intersection is zero. 
Also, if $m=k-1$, then the intersection of any two distinct simplex codes containing $C$ coincides with $C$
(it was established above that there is more than one such simplex code).

Suppose that $s=k-m$. Then $s\in \{1,\dots,k\}$.
The inequality 
\begin{equation}\label{eq-qk}
q^{k-1}>s+1
\end{equation}
 holds for all natural $s\in \{1,\dots,k\}$ except the following cases:
\begin{enumerate}
\item[(a)] $q=k=2$ and $s=1,2$;
\item[(b)] $q=2$ and $k=s=3$;
\item[(c)] $q=3$ and $k=s=2$.
\end{enumerate}
Observe  that $q^{k-1}>k+1$ if $q=2$, $k=4$ or $q=3$, $k=3$ or $q\ge 4$, $k=2$ 
and prove the  inequality for each $q\ge 2$ and the suitable values of $k$ by induction on $k$.

The case (a) is impossible by our assumption. 
It was noted above that the statement holds in the cases (b) and (c). 
From this moment, we assume that $k\ge 4$ if $q=2$ and $k\ge 3$ if $q=3$.
Then \eqref{eq-qk} holds for each $s\in \{1,\dots,k\}$.
Since the statement is true for $s=1$ ($m=k-1$),
we restrict ourself to the case when $s\in \{2,\dots,k\}$.

Suppose that $s<k$ ($m>0$) and $S$ is a generator matrix of $C$.
By Proposition \ref{matrix},  $S$ contains precisely $[s]_q$ zero columns. 
Without loss of generality we assume that the first $[s]_q$ columns of $S$ are zero.
Since $s\ge 2$, we have
$$[s]_q\ge s+1$$
(easy verification).
As in the proof of Proposition \ref{matrix},
we construct generator matrices  $M$ and $M'$ of two simplex codes
such that the first $m$ rows of these matrices coincides with the rows of $S$ and
the remaining rows $\bm{x}_1,\dots,\bm{x}_s$ of $M$ and $\bm{x}'_1,\dots,\bm{x}'_s$ of $M'$ satisfy 
\begin{equation}\label{eq-MM}
\left[\begin{array}{c}
\bm{x}_1\\
\bm{x}_2\\
\vdots \\
\bm{x}_s
\end{array}
\right]=\left[\begin{array}{ccc}
\bm{v}^t&I_s&A\\
\end{array}
\right]\;\mbox{ and }\;
\left[\begin{array}{c}
\bm{x}'_1\\
\bm{x}'_2\\
\vdots \\
\bm{x}'_s
\end{array}
\right]=\left[\begin{array}{ccc}
I_s&\bm{v}^t&A\\
\end{array}
\right],
\end{equation}
where $\bm{v}=(1,1,0,0,\ldots, 0)$ and $A$ is a certain $[s\times (n-s-1)]$-matrix.
If the simplex code associated to $M$ contains a non-zero linear combination
$$a_1\bm{x}'_1+\dots+a_s\bm{x}'_s,$$
then it also contains
$$a_1(\bm{x}_1-\bm{x}'_1)+\dots+a_s(\bm{x}_s-\bm{x}'_s)$$
which is
$$(a_2,a_1-a_2,a_2-a_3,\ldots,a_{s-1}-a_s,a_s-a_1-a_2,0,\ldots,0).$$
This vector is nonzero (otherwise, each of $a_1,\ldots, a_s$ is zero which is impossible by assumption) and
it's Hamming weight is not greater than $s+1$. By \eqref{eq-qk}, the vector cannot be a code word of a simplex code.
Therefore, the intersection of the simplex codes defined by $M$ and $M'$ coincides with $C$.

In the case when $s=k$ ($m=0$),
we  construct two generator matrices of simplex codes whose rows
$\bm{x}_1,\dots,\bm{x}_k$ and $\bm{x}'_1,\dots,\bm{x}'_k$ (respectively) satisfy \eqref{eq-MM} for $s=k$.
As above, we show that 
the intersection of the associated simplex codes is zero.
\end{proof}

\begin{rem}{\rm
Note that the above reasonings do not work in the cases when $q=2$, $k=3$ and $q=3$, $k=2$.
We are not able to find arguments covering all cases simultaneously.
}\end{rem}

\section{The graph of simplex codes}
The {\it Grassmann graph} $\Gamma_k(V)$ is the simple graph whose points are $k$-dimensional subspaces (codes) of $V$
and two such subspaces are connected by an edge, in other words, are adjacent  if their intersection is $(k-1)$-dimensional. 
We assume that $n=[k]_q$, $k\ge 2$ and denote by $\Gamma^{s}(k,q)$ the subgraph of $\Gamma_k(V)$ induced by 
the set of $q$-ary simplex codes of dimension $k$.

Consider a few examples:
\begin{enumerate}
\item[$\bullet$] $\Gamma^{s}(2,2)$ is a single vertex;
\item[$\bullet$] $\Gamma^{s}(2,3)$ is the complete bipartite graph $K_{4,4}$, see \cite[Example 2]{KP-simp2};
\item[$\bullet$] $\Gamma^{s}(3,2)$ is the graph $\Gamma_{1,3}({\mathbb F}^4_2)$
formed by $1$-dimensional and $3$-dimensional subspaces of ${\mathbb F}^4_2$, where distinct subspaces are connected by an edge if
they are incident \cite[Corollary 2]{KPP}.
\end{enumerate}

\begin{prop}\label{prop-conn}
The graph $\Gamma^{s}(k,q)$ is connected.
\end{prop}

\begin{proof}
The case when $q=k=2$ is trivial and we assume that $k\ne 2$ or $q\ne 2$.
We will need the following observations:
\begin{itemize}
\item[$\bullet$] if we multiply a column in a generator matrix of a simplex code by any $a\in {\mathbb F}\setminus\{0,1\}$,
then we obtain a generator matrix of an adjacent simplex code;
\item[$\bullet$] if we transpose any two columns in a generator matrix of a simplex code,
then we obtain a generator matrix of an adjacent simplex code.
\end{itemize}
Let $\bm{w}\in {\mathbb F}^k$ and let $$M=\left[\begin{array}{cc}
\bm{w}^t&A\\
\end{array}
\right],\;\; M'=\left[\begin{array}{cc}
a\bm{w}^t&A\\
\end{array}
\right]$$
be generator matrices of simplex codes $C\subset V$ and $C'\subset V$, respectively.
There is a non-singular $k$-matrix $T$ such that $T\bm{w}^t=(1,0,0,\cdots,0)^t$.
Then
$$TM=\left[\begin{array}{cc}
1&\\
0&\\
0&B\\
\vdots&\\
0&\\
\end{array}
\right],\;\; TM'=\left[\begin{array}{cc}
a&\\
0&\\
0&B\\
\vdots&\\
0&\\
\end{array}
\right]$$
are generator matrices of $C$ and $C'$, respectively.
These matrices are different in the first rows only,
these rows are non-proportional and their difference is of Hamming weight $1$.
This means that $C,C'$ are distinct and, consequently they are adjacent.

Now, let $\bm{w},\bm{u}\in {\mathbb F}^k$ and let
$$M=\left[\begin{array}{ccc}
\bm{w}^t&\bm{u}^t&A\\
\end{array}
\right],\;\; M'=\left[\begin{array}{ccc}
\bm{u}^t&\bm{w}^t&A\\
\end{array}
\right]$$
be generator matrices of simplex codes $C\subset V$ and $C'\subset V$, respectively.
We take a non-singular $k$-matrix $T$ such that 
$$T\bm{w}^t=(1,1,0,\cdots,0)^t\;\mbox{ and }\;T\bm{u}^t=(0,1,0,\cdots,0)^t.$$
Then
$$TM=\left[\begin{array}{ccc}
1&0&\\
1&1&\\
0&0&B\\
\vdots&\vdots& \\
0&0&\\
\end{array}
\right],\;\; TM'=\left[\begin{array}{ccc}
0&1&\\
1&1&\\
0&0&B\\
\vdots&\vdots& \\
0&0&\\
\end{array}
\right]$$
are generator matrices of $C$ and $C'$, respectively.
These matrices are different in the first rows only, these rows are non-proportional and their difference is of Hamming weight $2$.
As above, we obtain that $C\ne C'$ and our simplex codes are adjacent.

Any generator matrix of a simplex code  can be obtained from a generator matrix of any other simplex code by 
a series of the operations described above. This implies the connectedness of $\Gamma^{s}(k,q)$.
\end{proof}

\begin{rem}{\rm
We are not able to determine the path distance between vertices of $\Gamma^{s}(k,q)$. This problem looks difficult. 
}\end{rem}

Recall that a {\it clique} is a subset in the vertex set of a graph, where any two distinct vertices are connected by an edge. 
A clique ${\mathcal X}$ is said to be {\it maximal} if every clique containing ${\mathcal X}$ coincides with ${\mathcal X}$.
It is well-known that every maximal clique of $\Gamma_k(V)$ is of one of the following type:
\begin{enumerate}
\item[$\bullet$] the {\it star} ${\mathcal S}(X)$ consisting of all $k$-dimensional subspaces containing a certain $(k-1)$-dimensional subspace $X$;
\item[$\bullet$]  the {\it top} ${\mathcal T}(Y)$ consisting of all $k$-dimensional subspaces contained in a certain $(k+1)$-dimensional subspace $Y$.
\end{enumerate}
The intersections of  ${\mathcal S}(X)$ and ${\mathcal T}(Y)$ with the set of simplex codes are denoted by ${\mathcal S}^s(X)$ and ${\mathcal T}^s(Y)$,
respectively. Every such intersection is a clique in $\Gamma^{s}(k,q)$ (if it is non-empty), but we cannot assert that this clique is maximal.
We say that ${\mathcal S}^s(X)$ or ${\mathcal T}^s(Y)$ is a {\it star} or a {\it top} of the simplex code graph $\Gamma^{s}(k,q)$
only in the case when it is a maximal cliques of $\Gamma^{s}(k,q)$. 
It is clear that every maximal clique of $\Gamma^{s}(k,q)$ is a star or a top.

Since $\Gamma^{s}(2,3)$ and $\Gamma^{s}(3,2)$ are  bipartite, every maximal clique in these graphs consists of two vertices
which implies that it is a star and a top simultaneously. Indeed, if $X,Y$ are adjacent vertices in one of these graphs, then
$$\{X,Y\}={\mathcal S}^s(X\cap Y)={\mathcal T}^s(X+Y).$$

\begin{prop}\label{prop-star}
Suppose that one of the following possibilities is realized:
\begin{enumerate}
\item[$\bullet$] $q=2$ and $k\ge 4$;
\item[$\bullet$] $q=3$ and $k\ge 3$;
\item[$\bullet$] $q\ge 4$.
\end{enumerate}
Then ${\mathcal S}^s(X)$ is a star of $\Gamma^{s}(k,q)$ if and only if $X$ is a $(k-1)$-dimensional subcode of a simplex code.
Furthermore, there are no maximal cliques of $\Gamma^{s}(k,q)$ which are stars and tops simultaneously.
\end{prop}

\begin{proof}
It is clear that ${\mathcal S}^s(X)$ is non-empty  if and only if $X$ is a $(k-1)$-dimensional subcode of a simplex code.
By Theorem \ref{extension},  
$$|{\mathcal S}^s(X)|=\frac{(q!)^{[k-1]_q}}{q^{k-1}}$$
(we take $m=k-1$ in \eqref{numb2}).
We assert that this number is greater than $q+1$.

If $k=2$, then $q\ge 4$ (by assumption). In this case, ${\mathcal S}^s(X)$ consists of $(q-1)!$ elements and
$$(q-1)!\ge (q-2)(q-1)=q^2-3q+2>q+1$$
(since $q^2-4q+1>0$ for $q\ge 4$).
Suppose that $k\ge 3$. Then $$[k-1]_q\ge k+1.$$
Indeed, 
$$[k-1]_q=q^{k-2}+q^{k-3}\dots+1\ge q^{k-2}+k-2\ge k+1$$
(we have $q^{k-2}\ge 3$, since $k\ge 4$ if $q=2$).
Therefore,
$$\frac{(q!)^{[k-1]_q}}{q^{k-1}}\ge \frac{q^{k+1}}{q^{k-1}}=q^2>q+1.$$
So, ${\mathcal S}^s(X)$ contains more than $q+1$ elements
and there is no ${\mathcal T}^s(Y)$ containing ${\mathcal S}^s(X)$
(since the intersection of a star and a top of the Grassmann graph is empty or contains precisely $q+1$ elements).
This guarantees that ${\mathcal S}^s(X)$ is a maximal clique of  $\Gamma^{s}(k,q)$, i.e. it  is a star of  $\Gamma^{s}(k,q)$
which is not a top.
\end{proof}

We will discuss the following question: for which cases mentioned in Proposition \ref{prop-star} the graph $\Gamma^{s}(k,q)$
contains tops. 
The following example shows that there is non-empty  ${\mathcal T}^s(Y)$ which is not a top of $\Gamma^{s}(k,q)$.

\begin{exmp}{\rm
Suppose that $q\ge 4$, $k=2$ and $\alpha$ is a primitive element of ${\mathbb F}={\mathbb F}_q$.
Let us take $\bm{x},\bm{y},\bm{z}\in V$ such that 
$$\left[\begin{array}{c}
\bm{x}\\
\bm{y}\\
\bm{z}
\end{array}
\right]=\left[\begin{array}{ccccccc}
0&1&1&\alpha&\alpha^2&\cdots &\alpha^{q-2}\\
1&0&1&1&1&\cdots&1\\
1&0&\alpha&\alpha&\alpha&\cdots &\alpha\\
\end{array}
\right].$$
Then $\langle \bm{x},\bm{y} \rangle$ and $\langle \bm{x},\bm{z}\rangle$ are adjacent $q$-ary simplex codes of dimension $2$.
Assume that there is a simplex code  intersecting $\langle \bm{x},\bm{y} \rangle$ and $\langle \bm{x},\bm{z}\rangle$ in distinct 
$1$-dimensional subcodes $\langle \bm{x}+\alpha^j\bm{z} \rangle, \langle \bm{x}+\alpha^i \bm{y} \rangle$, respectively. Its generator matrix is
$$\left[\begin{array}{c}
\bm{x}+\alpha^j\bm{z}\\
\bm{x}+\alpha^i\bm{y}\\
\end{array}
\right]=\left[\begin{array}{ccccccc}
\alpha^j&1&1+\alpha^{j+1}&\alpha+\alpha^{j+1}&\alpha^2+\alpha^{j+1}&\cdots &\alpha^{q-2}+\alpha^{j+1}\\
\alpha^i&1&1+\alpha^{i}&\alpha+\alpha^{i}&\alpha^2+\alpha^{i}&\cdots &\alpha^{q-2}+\alpha^{i}\\
\end{array}
\right]$$
and $\alpha^i\neq \alpha^j$ (since the first and second columns are non-proportional). 
We choose $t\in\{0,1,\dots,q-2\}$ such that
$$\alpha^t=\frac{\alpha^i(\alpha^{j+1}-\alpha^j)}{\alpha^j-\alpha^i}.$$
Then the determinant 
$$\left|\begin{array}{cc}
\alpha^j&\alpha^t+\alpha^{j+1}\\
\alpha^i&\alpha^t+\alpha^i\\
\end{array}
\right|=\alpha^t(\alpha^j-\alpha^i)-\alpha^i(\alpha^{j+1}-\alpha^j)$$
is zero, i.e. the first and $(t+3)$-th columns are proportional which is impossible. 
Therefore, every simplex code adjacent to both $\langle \bm{x},\bm{y} \rangle,\langle \bm{x},\bm{z}\rangle$ belongs to the star 
${\mathcal S}^s(\langle \bm{x}\rangle)$.
Then ${\mathcal T}^s(\langle \bm{x},\bm{y},\bm{z}\rangle)$ is a non-empty proper subset of ${\mathcal S}^s(\langle \bm{x}\rangle)$
and, consequently, it is not a top.
}\end{exmp}

Recall that every maximal clique of $\Gamma^s(2,3)$ and $\Gamma^s(3,2)$ is a star and a top simultaneously. 
Every maximal clique of $\Gamma^s(2,4)$ is a star \cite[Proposition 3]{KP-simp2}. 
We show that  $\Gamma^s(k,q)$ contains tops in all remaining cases (by Proposition \ref{prop-star}, a maximal clique cannot be a star and a top simultaneously for 
each of these cases); furthermore, for many cases there are tops containing different numbers of elements. 

\begin{theorem}\label{theorem-tops}
Suppose that one of the following possibilities is realized:
\begin{enumerate}
\item[$\bullet$] $k=2$ and $q\ge 5$,
\item[$\bullet$] $k\ge 4$ and  $q=2$,
\item[$\bullet$] $k\ge 3$ and $q\ge 3$.
\end{enumerate}
Then $\Gamma^s(k,q)$ contains tops.
If $k\ge 4$ and $q\ge 3$, then there are tops of $\Gamma^s(k,q)$ containing different numbers of elements.
If $k\ge 5$ and $q\ge 3$, then there is a top of $\Gamma^s(k,q)$ consisting of precisely three elements.
\end{theorem}

\section{Proof of Theorem \ref{theorem-tops}}
\subsection{The case $k=2$ and $q\ge 5$}
In this case, $n=q+1\ge 6$ and the Hamming weight of every non-zero code word in a simplex code is equal to $q$.
Let $\alpha$ be a primitive element of the field.
We find linearly independent vectors $\bm{x},\bm{y},\bm{z}\in V$
such that $\langle \bm{x},\bm{y}\rangle, \langle \bm{x},\bm{z}\rangle, \langle \bm{y},\bm{z}\rangle$ are simplex codes which means that
${\mathcal T}^s(\langle \bm{x},\bm{y},\bm{z}\rangle)$ is a top of $\Gamma^s(k,q)$.

Consider $\bm{x},\bm{y},\bm{z}\in V$ such that
$$\left[\begin{array}{c}
\bm{x}\\
\bm{y}\\
\bm{z}\\
\end{array}
\right]=\left[\begin{array}{cccccccc}
0&1&1&1&1&\cdots &1&1\\
1&0&\alpha^0&\alpha^1 &\alpha^2 &\cdots &\alpha^{q-3}&\alpha^{q-2}\\
1&\frac{1}{1+0}&\frac{1}{1+\alpha^0}&\frac{1}{1+\alpha^{-1}}&\frac{1}{1+\alpha^{-2}}&\cdots &\frac{1}{1+\alpha^{-(q-3)}}&\frac{1}{1+\alpha^{-(q-2)}}\\
\end{array}
\right];$$
if $1+\alpha^{-i}=0$, then we put $0$ instead of $\frac{1}{1+\alpha^{-i}}$.
It is clear that the columns of $\genfrac[]{0pt}{2}{\bm{x}}{\bm{y}}$ are mutually non-proportional.
Since the functions $f(x)=1+x$ and $$g(x)=\left\{\begin{array}{cc}
x^{-1}& \;\; x\neq 0\\
0 & \;\; x=0\\
\end{array}\right.$$
are bijective, the same holds for the columns of $\genfrac[]{0pt}{2}{\bm{x}}{\bm{z}}$.
So, $\left\langle \bm{x},\bm{y}\right\rangle$ and $\left\langle \bm{x},\bm{z}\right\rangle$ are simplex codes. 
We check that the columns of
$$\left[\begin{array}{c}
\bm{y}\\
\bm{z}\\
\end{array}
\right]=\left[\begin{array}{cccccccc}
1&0&\alpha^0&\alpha^1 &\alpha^2 &\cdots &\alpha^{q-3}&\alpha^{q-2}\\
1&\frac{1}{1+0}&\frac{1}{1+\alpha^0}&\frac{1}{1+\alpha^{-1}}&\frac{1}{1+\alpha^{-2}}&\cdots &\frac{1}{1+\alpha^{-(q-3)}}&\frac{1}{1+\alpha^{-(q-2)}}\\\end{array}
\right]$$
are mutually non-proportional. 

Observe that $(-1, 0)^t$ is a column of $\genfrac[]{0pt}{2}{\bm{y}}{\bm{z}}$;
this is the third column if $q$ is even;
in the case when $q$ is odd, this column corresponds  to $\alpha^{\frac{q-1}{2}}=-1$.
It is clear that the columns $(0,1)^t$ and $(-1,0)^t$
are non-proportional to any other column.
Also, the columns 
$$\left(\alpha^i, \frac{1}{1+\alpha^{-i}}\right)^t$$
are non-proportional to the first column
(the equality $\alpha^{i}=\frac{1}{1+\alpha^{-i}}$ implies that $\alpha^{i}=0$).
It remains to show that 
$$\left(\alpha^i, \frac{1}{1+\alpha^{-i}}\right)^t,\;\left(\alpha^j, \frac{1}{1+\alpha^{-j}}\right)^t$$
are non-proportional for any distinct $i,j$ such that
$1+\alpha^{-i}$, $1+\alpha^{-j}$ both are non-zero.
The determinant 
$$\left|\begin{array}{cc}
\alpha^i&\alpha^j\\
\frac{1}{1+\alpha^{-i}}&\frac{1}{1+\alpha^{-j}}\\
\end{array}
\right|=\frac{\alpha^i}{1+\alpha^{-j}}-\frac{\alpha^j}{1+\alpha^{-i}}$$
is equal to $0$ if and only if 
$$\alpha^i(1+\alpha^{-i})=\alpha^j(1+\alpha^{-j})$$
or, equivalently, if and only if $\alpha^i+1=\alpha^j+1$
which is possible only in the case when $i=j$.

Show that $\bm{x},\bm{y},\bm{z}$ are linearly independent. 
If the field characteristic is not $2$ or $3$, then 
the vectors
$$(0,1,1),\; (1,0,1),\; (1,1,2^{-1})$$
formed by the first three coordinates of $\bm{x},\bm{y},\bm{z}$ (respectively)
are linearly independent (since $2\ne 2^{-1}$).
In the case of characteristic $2,3$ and $q\neq 2,3,4$, we consider the vectors
$$(0,1,1),\; (1,0,\alpha),\; \left(1,1,\frac{1}{1+\alpha^{-1}}\right )$$
formed by the first, second and fourth coordinates of $\bm{x},\bm{y},\bm{z}$ , respectively. 
If these vectors are linearly dependent, then $1+\alpha=\frac{1}{1+\alpha^{-1}}$ which implies that
$$\alpha+\alpha^{-1}+2=1,$$
$$\alpha+1+\alpha^{-1}=0,$$
$$\alpha^2+\alpha+1=0.$$
For the characteristic $2$ only the primitive element of $\mathbb{F}_4$ satisfies the latter equality. 
In the case of characteristic $3$, the polynomial 
$$x^2+x+1=x^2-2x+1=(x-1)^2$$ is reducible
and there is no primitive element $\alpha$ of $\mathbb{F}_{3^m}$, $m\geq2$ satisfying 
the above equality.

\subsection{The case  $k\ge 4$ and $q=2$} Let
$$A=\left[\begin{array}{ccc}
0&1&1\\
1&0&1\\
1&1&0
\end{array}\right]
\;\mbox{ and }\;
B=\left[\begin{array}{cccc}
1&0&0&1\\
0&1&0&1\\
0&0&1&1\\
\end{array}
\right]$$
(note that the rows of $A$ are non-zero code words of the binary simplex code of dimension $2$).
In this case, the vector space ${\mathbb F}^{k-2}$ contains precisely $[k-2]_2$ non-zero vectors which we denote by
$\bm{w}_1,\dots, \bm{w}_{[k-2]_2}$.
For every $i\in \{1,\dots,[k-2]_2\}$ we denote by $D_i$ the $(k-2,4)$-matrix whose columns are $\bm{w}_i$.
Observe that $$n=[k]_2=4[k-2]_2 +3$$ and consider the $(k+1,n)$-matrix
$$M=\left[\begin{array}{cccc}
A&B&\dots&B\\
\bm{0}&D_1&\dots &D_{[k-2]_2}
\end{array}
\right].
$$
Let $\bm{v}_1,\bm{v}_2$ and $\bm{v}_3$ be the first, second and third rows of $M$ (respectively).
Removing $\bm{v}_i$, $i\in \{1,2,3\}$ from $M$ we obtain a $(k,n)$-matrix whose columns are non-zero and mutually distinct,
i.e. a generator matrix of a certain simplex code $X_i$. 
Denote by $C$ the $(k-2)$-dimensional subspace whose generator matrix is
$[\bm{0},D_1,\dots, D_{[k-2]_2}]$. 

Then $C$ is a subcode of every $X_i$.
The vectors $\bm{v}_1,\bm{v}_2,\bm{v}_3$ are linearly independent and
we assert that $\langle \bm{v}_1,\bm{v}_2,\bm{v}_3\rangle$ intersects $C$ precisely in zero.

It is clear that each $\bm{v}_i$ and the sum of any two distinct $\bm{v}_i$ do not belong to $C$.
The Hamming weight of
$$\bm{v}_1+\bm{v}_2+\bm{v}_3=(0,0,0,1,\dots,1)$$
is $n-3=2^{k}-4\ne 2^{k-1}$ (we have $k\ge 4$, for $k=3$ this fails).
Since every non-zero vector from $C$ is of Hamming weight $2^{k-1}$, we obtain that $\bm{v}_1+\bm{v}_2+\bm{v}_3\not\in C$.

So, the rows of $M$ are linearly independent and $X_1,X_2,X_3$ are mutually distinct.
Therefore, $X_1,X_2,X_3$ are adjacent and their intersection is $C$.
The maximal clique of $\Gamma^s(k,q)$ containing them is a top.

\subsection{The case when $q\ge 3$ and $k\ge 3$}
As in Subsection 5.1,  $\alpha$ is a primitive element of the field. 

We take $\bm{x},\bm{y}\in {\mathbb F}^{q+1}$ spanning a $q$-ary simplex code of dimension $2$.
For any non-zero scalars $a,b\in {\mathbb F}$ the $(3,q+1)$-matrix
$$\left[\begin{array}{c}
\bm{x}\\
\bm{y}\\
a\bm{x}+b\bm{y}
\end{array}\right]$$
will be denote by $A_{a,b}$.
Observe that $$q^2=(q-1)(q+1)+1$$ and denote by $B_{a,b}$ the $(3,q^2)$-matrix 
$$[A_{a,b}\; \alpha A_{a,b}\;\dots\;\alpha^{q-2}A_{a,b}\;\bm{0}]$$
obtained by joining  of all matrices proportional to $A_{a,b}$ and adding zero column.

There are precisely $[k-2]_q$ non-zero vectors of ${\mathbb F}^{k-2}$ whose first non-zero coordinate is $1$.
Let $\bm{w}_1,\dots, \bm{w}_{[k-2]_q}$ be these vectors. 
For every $i\in \{1,\dots,[k-2]_q\}$ we denote by $D_i$ the $(k-2,q^2)$-matrix whose columns are $\bm{w}_i$.
Let us take any collection of non-zero scalars
$$a_0,b_0,a_1,b_1,\dots,a_{[k-2]_q},b_{[k-2]_q}\in {\mathbb F}$$
satisfying the following condition:
\begin{enumerate}
\item[(I)]for at least two distinct $i,j\in \{0,1,\dots, [k-2]_q\}$ we have
$a_i\ne a_j$ or $b_i\ne b_j$.
\end{enumerate}
Observe that $$n=[k]_q=q^2[k-2]_q +q+1$$ and consider the $(k+1,n)$-matrix
$$M=\left[\begin{array}{cccc}
A_{a_0,b_0}&B_{a_1,b_1}&\dots&B_{a_{[k-2]_q},b_{[k-2]_q}}\\
\bm{0}&D_1&\dots &D_{[k-2]_q}
\end{array}
\right].$$
The first, second and third rows of $M$ will be denoted by $\bm{v}_1,\bm{v}_2$ and $\bm{v}_3$, respectively.
The condition (I) guarantees that these vectors are linearly independent. 
Removing one of $\bm{v}_i$, $i\in \{1,2,3\}$ from $M$ we  obtain 
a $(k,n)$-matrix. A direct verification shows that the columns of this matrix are non-zero
and mutually non-proportional, i.e. it is a generator matrix of a certain simplex code $X_i$.

Let $C$ be the $(k-2)$-dimensional subspace of $V$ whose generator matrix is
$$[\bm{0}\; D_1\;\dots\;D_{[k-2]_q}].$$
Then $C$ is a subcode of every $X_i$.
To assert that $X_1,X_2,X_3$ are mutually distinct we need to show that $C$ does not contain
a non-zero linear combination of $\bm{v}_1,\bm{v}_2,\bm{v}_3$.
If $a\bm{v}_1+b\bm{v}_2+c\bm{v}_3$ belongs to $C$, then the first $q+1$ coordinates of $a\bm{v}_1+b\bm{v}_2+c\bm{v}_3$ are zero 
which means that
$$a\bm{x}+b\bm{y}+c(a_0\bm{x}+b_0\bm{y})=\bm{0}\;\mbox{ and }\;a=-ca_0,\;b=-cb_0$$
(since $\bm{x},\bm{y}$ are linearly independent); in particular, $c\ne 0$.
Without loss of generality we assume that $c=-1$ and consider 
the vector $a_0\bm{v}_1+b_0\bm{v}_2-\bm{v}_3$.
Every non-zero linear combination of $\bm{x},\bm{y}$ is of Hamming weight $q$ (as a code word of a $q$-ary simplex code of dimension $2$). 
Observe that $a_0\bm{v}_1+b_0\bm{v}_2-\bm{v}_3$ is obtained by joining of at most $(q-1)[k-2]_q$ such linear combinations and adding some  zero coordinates. 
This means that the Hamming weight of $a_0\bm{v}_1+b_0\bm{v}_2-\bm{v}_3$ is not greater than 
$$q(q-1)[k-2]_q=q^{k-1}-q,$$
i.e. $a_0\bm{v}_1+b_0\bm{v}_2-\bm{v}_3$ is not a code word of a simplex code and, consequently, it does not belong to $C$.

So, $X_1,X_2,X_3$ are mutually distinct and, consequently, they are adjacent and their intersection is $C$.
The maximal clique of $\Gamma^s(k,q)$ containing $X_1,X_2,X_3$ is the top corresponding to
the $(k+1)$-dimensional subspace $S$ whose generator matrix is $M$.

\begin{lemma}\label{lemma-1}
A non-zero vector of $\langle \bm{v}_1,\bm{v}_2, \bm{v}_3\rangle$ is not a code word of a simplex code if and only if 
it is a scalar multiple of 
$$a_i\bm{v}_1+b_i\bm{v}_2-\bm{v}_3$$
for a certain $i\in \{0,1,\dots, [k-2]_q\}$.
\end{lemma}

\begin{proof}
We determine all non-zero vectors of $\langle \bm{v}_1, \bm{v}_2, \bm{v}_3\rangle$ whose Hamming weight is not equal to $q^{k-1}$.
If  $\bm{v}=a\bm{v}_1+b\bm{v}_2+c\bm{v}_3$ is such a vector,  then each of the scalars $a,b,c$ is non-zero (otherwise, $\bm{v}$ belongs to one of $X_i$ which contradicts 
our assumption). So, we can suppose that $c=-1$.
For every $i\in \{0,1,\dots, [k-2]_q\}$ we define
$$\bm{u}_i=a\bm{x}+b\bm{y}-a_i \bm{x}-b_i \bm{y}=(a-a_i)\bm{x}+(b-b_i)\bm{y}.$$
Then $v$ is obtained by joining  of all vector proportional to all $\bm{u}_i$ and adding some zero coordinates.
This means that the Hamming weight of $\bm{v}$ is
$$h_0+(q-1)\sum^{[k-2]_q}_{i=1}h_i,$$
where $h_i$, $i\in \{0,1,\dots, [k-2]_q\}$ is the Hamming weight of $\bm{u}_i$.
Since each $\bm{u}_i$ is a linear combination of $\bm{x},\bm{y}$, i.e. a code word of the $2$-dimensional $q$-ary simplex code $\langle \bm{x},\bm{y}\rangle$,
the Hamming weight of $\bm{u}_i$ is $q$ or $0$.
In the case when each $h_i$ is $q$, the Hamming weight of $\bm{v}$
is $$q+q(q-1)[k-2]_q=q^{k-1}$$
and $\bm{v}$ is a code word of a simplex code.
If a certain $h_i$ is zero, then the Hamming weight of $\bm{v}$ is less than $q^{k-1}$
and $\bm{u}_i=\bm{0}$.
The latter implies that $a=a_i$, $b=b_i$.
\end{proof}

From this moment, we assume that our collection $\{a_i,b_i\}^{[k-2]_q}_{i=1}$ satisfy the following more strong condition:
\begin{enumerate}
\item[(II)]for every $i\in \{1,\dots, [k-2]_q\}$ we have
$a_i\ne a_0$ or $b_i\ne b_0$.
\end{enumerate}

\begin{lemma}\label{lemma-2}
If {\rm (II)} holds, then the following assertions are fulfilled:
\begin{enumerate}
\item[(1)] Every vector belonging to $\langle C, a_0\bm{v}_1+b_0\bm{v}_2-\bm{v}_3 \rangle \setminus C$
is not a code word of a simplex code.
\item[(2)] Every simplex code from ${\mathcal T}^s(S)$ contains $C$.
\item[(3)]  If $v\in \langle \bm{v}_1,\bm{v}_2,\bm{v}_3\rangle$ is a code word of a simplex code, then the same holds for every vector belonging to
 $\langle C, \bm{v} \rangle$.
\item[(4)] A $k$-dimensional subspace of $S$ belongs to ${\mathcal T}^s(S)$ if and only if 
it is the sum of $C$ and a $2$-dimensional subspace of $\langle \bm{v}_1,\bm{v}_2,\bm{v}_3\rangle$ which is a subcode of a simplex code.
\end{enumerate}
\end{lemma}

\begin{proof}
Recall that $S$ is the sum of $C$ and $\langle \bm{v}_1,\bm{v}_2,\bm{v}_3\rangle$.
Let $\bm{v}\in \langle \bm{v}_1,\bm{v}_2,\bm{v}_3\rangle$.
Then $\bm{v}=a\bm{v}_1+b\bm{v}_2+c\bm{v}_3$.
In the case when $a,b,c$ are non-zero, we assume that $c=-1$ and take
$\bm{u}_i$, $i\in \{0,1,\dots, [k-2]_q\}$ as in the proof of the previous lemma.

(1). Suppose that $\bm{v}=a_0\bm{v}_1+b_0\bm{v}_2-\bm{v}_3$. Then $\bm{u}_0=\bm{0}$. 
The condition (II) guarantees that any other $\bm{u}_i$ is non-zero and, consequently, its Hamming weight is $q$.
If $\bm{c}\in C$, then
$$\bm{c}=(\underbrace{0,\ldots,0}_{q+1},\underbrace{d_1,\ldots,d_1}_{q^2},\dots,\underbrace{d_{[k-2]_q},\ldots,d_{[k-2]_q}}_{q^2});$$
since the Hamming weight of $\bm{c}$ is $q^{k-1}$, $d_i$ is non-zero for precisely $q^{k-3}$ distinct $i$.
We need to show that the Hamming weight of $\bm{v}+\bm{c}$ is not equal to $q^{k-1}$.

For every $i\ge 1$ the vector 
$$\tilde{\bm{u}}_i=(\bm{u}_i,\alpha \bm{u}_i,\ldots \alpha^{q-2}\bm{u}_i,0)\in {\mathbb F}^{q^2}$$
obtained by joining of all non-zero scalar multiples of $\bm{u}_i$ and adding the zero coordinate is of
Hamming weight $q(q-1)$.
Since the first $q+1$ coordinates of $\bm{v}+\bm{c}$ are zero,
the Hamming weight of $\bm{v}+\bm{c}$ is the sum of the Hamming weights of the vectors 
$$\tilde{\bm{u}}_i+(d_i,\dots, d_i),\;\;\;i\in \{1,\dots, [k-2]_q\}.$$
This vector is of Hamming weight $q(q-1)$ if $d_i=0$.
Observe that for every non-zero $t\in {\mathbb F}$ the elements $t,\alpha t,\dots, \alpha^{q-2} t$ are non-zero and mutually distinct.
This implies that for every $d_i\ne 0$ the vector $\tilde{\bm{u}}_i$
contains precisely $q$ coordinates equal to $-d_i$; since this vector contains precisely $q$ zero coordinates,
the Hamming weight of $\tilde{\bm{u}}_i+(d_i,\dots, d_i)$ is also $q(q-1)$.
Therefore, the Hamming weight of $\bm{c}+\bm{v}$ is
$$[k-2]_q(q-1)q=q^{k-1}-q$$
which means that $\bm{v}+\bm{c}$ is not a code word of a simplex code.

(2). Every simplex code from ${\mathcal T}^s(S)$ intersects the $(k-1)$-dimensional subspace $\langle C, \bm{v}_1+\bm{v}_2-\bm{v}_3\rangle$ 
in a subspace whose dimension is not less than $k-2$.
By the statement (1), this intersection does not contain vectors from $\langle C, \bm{v}_1+\bm{v}_2-\bm{v}_3\rangle\setminus C$ and,
consequently, it coincides with $C$.

(3). Suppose that $\bm{v}=a\bm{v}_1+b\bm{v}_2+c\bm{v}_3$ is a code word of a simplex code and $\bm{c}\in C$.
If one of $a,b,c$ is zero, then $\bm{v}+\bm{c}$ belongs to one of the simplex codes $X_i$, $i\in \{1,2,3\}$. 
Consider the case when  these scalars are non-zero. Then $c=-1$ by our assumption and all $\bm{u}_i$ are non-zero. 
Since $\bm{u}_0\ne \bm{0}$, only one of the first $q+1$ coordinates of $\bm{v}+\bm{c}$ is zero. By the reasoning from the proof of the statement  (1),
precisely $q^{k-1}-q$ of the remaining coordinates of $\bm{v}+\bm{c}$ are non-zero.
So, the Hamming weight of $\bm{c}+\bm{v}$ is $q+q^{k-1}-q=q^{k-1}$ and $\bm{c}+\bm{v}$ is a code word of a simplex code.

(4). Every simplex code from ${\mathcal T}^s(S)$ intersects $\langle \bm{v}_1,\bm{v}_2,\bm{v}_3\rangle$
in a $2$-dimensional subspace and it is the sum of $C$ and this subspace. 
If $X\subset \langle \bm{v}_1,\bm{v}_2,\bm{v}_3\rangle$ is a $2$-dimensional subcode of a simplex code,
then $C+X$ is $k$-dimensional and, by the statement (3), every non-zero vector from $C+X$ is of Hamming weight $q^{k-1}$
which means that $C+X$ is a simplex code belonging to ${\mathcal T}^s(S)$.
\end{proof}

To complete the proof of Theorem \ref{theorem-tops} we observe that there are precisely $(q-1)^2$ distinct pars of non-zero elements of the field. 

If $k=3$, then $[k-2]_q=1$ and  our collection consists of four elements $a_0,b_0,a_1,b_1$
such that $a_0\ne a_1$ or $b_0\ne b_1$.
This means that $\langle \bm{v}_1,\bm{v}_2,\bm{v}_3\rangle$ always contains precisely two $1$-dimensional subspaces
which are not subcodes of simplex codes. For this reason, we do not obtain tops of $\Gamma^s(k,q)$ containing different numbers of elements.

If $k\ge 4$, then $[k-2]_q\ge q+1$ and we can choose $\{a_i,b_i\}^{[k-1]_q}_{i=0}$ satisfying (II) in different ways:
such that there are precisely two distinct pairs $a_i,b_i$ or there are more than two distinct pairs. 
Therefore, $\langle \bm{v}_1,\bm{v}_2,\bm{v}_3\rangle$ can contain different numbers of $1$-dimensional subspaces
which are not subcodes of simplex codes and, consequently, there are tops of $\Gamma^s(k,q)$ containing different numbers of elements.

Suppose that $k\ge 5$. Then $$(q-1)^2<[k-2]_q.$$
Consider $\{a_i,b_i\}^{[k-1]_q}_{i=0}$ satisfying (II) and the following additional condition: for any non-zero (not necessarily distinct) $a,b\in {\mathbb F}$ 
there is $i$ such that $a_i=a$ and $b_i=b$.
Then  $\langle \bm{v}_1,\bm{v}_2,\bm{v}_3\rangle$ contains precisely $(q-1)^2$ distinct $1$-dimensional subspaces
which are not subcodes of simplex codes.
Each of the remaining 
$$q^2+q+1-(q-1)^2=3q+3$$
$1$-dimensional subspaces of $\langle \bm{v}_1,\bm{v}_2,\bm{v}_3\rangle$ is contained in one of $\langle \bm{v}_i,\bm{v}_j\rangle$, $i\ne j$.
By the statement (4) of Lemma \ref{lemma-2}, the top ${\mathcal T}^c(S)$ contains precisely three elements.

\section{Point-line geometries related to simplex codes}
A  {\it point-line geometry} is a pair $({\mathcal P},{\mathcal L})$,
where ${\mathcal P}$ is a set whose elements are called {\it points} and 
${\mathcal L}$ is a family of subsets of ${\mathcal P}$ called {\it lines}. 
Every line contains at least two points and the intersection of two distinct lines contains at most one point.
Two distinct points are said to be {\it collinear} if there is a line containing them.
The {\it collinearity graph} of $({\mathcal P},{\mathcal L})$
is the simple graph whose vertex set is ${\mathcal P}$ and two distinct vertices are connected by an edge if they are collinear points. 
A subset $X\subset {\mathcal P}$ is called a {\it subspace} if for any two collinear points from $X$
the line containing these points is a subset of $X$. 
We say that a subspace is {\it singular} if any two distinct points of this subspace are collinear. 

The projective space associated to $V$ is the point-line geometry whose points are $1$-dimensional subspaces of $V$
and whose lines correspond to $2$-dimensional subspaces of $V$.
In this projective space, we consider the subgeometry ${\mathcal S}(k,q)$  whose points are $1$-dimensional subcodes of simplex codes
and whose lines correspond to $2$-dimensional subcodes of simplex codes.
Then there is the natural one-to-one correspondence between singular subspaces of ${\mathcal S}(k,q)$ and subcodes of simplex codes;
in particular,  maximal singular subspaces correspond to simplex codes.

Every maximal singular subspace of ${\mathcal S}(k,q)$ is a maximal clique of the collinearity graph.
Indeed, if ${\mathcal X}$ is a maximal singular subspace of ${\mathcal S}(k,q)$ and 
$P\in {\mathcal S}(k,q)\setminus {\mathcal X}$ is a point collinear to all points of ${\mathcal X}$,
then we consider the simplex code $C$ corresponding to ${\mathcal X}$ and any non-zero $\bm{c}\in P$,
every non-zero vector from $\langle \bm{c},C\rangle$ is of Hamming weight $q^{k-1}$ which contradicts the fact 
that $C$ is a maximal $q^{k-1}$-equidistant code in $V$.

Maximal singular subspace of ${\mathcal S}(2,q)$ are lines. In the case when $q=3,4$, every clique of the collinearity graph is a line.
Indeed, every maximal clique of $\Gamma^s(2,3)$ is a star and a top simultaneously, 
every maximal clique of $\Gamma^s(2,4)$ is a star which is not a top; in each of these cases, three points of ${\mathcal S}(2,q)$ 
are mutually collinear if and only if they are on a common line.
Since ${\mathcal S}(3,2)$ is a polar space,
the family of maximal singular subspaces of ${\mathcal S}(3,2)$ coincides with the family of maximal cliques of the collinearity graph.
In the all remaining cases, there are maximal cliques of the collinearity graph
which are not maximal singular subspaces.

\begin{lemma}\label{tex}
Suppose that one of the conditions from Theorem {\rm \ref{theorem-tops}} is satisfied.
Let $\bm{x}$ and $\bm{y}$ be vectors of $V$ spanning a $2$-dimensional subcode of a simplex code.
For every vector $\bm{x}+a\bm{y}$, $a\ne 0$ there is a vector $\bm{z}\in V$ such that $\langle \bm{x},\bm{z}\rangle$ and $\langle \bm{y},\bm{z}\rangle$
are subcodes of simplex codes and there is no simplex code containing $\bm{z}$ and $\bm{x}+a\bm{y}$.
\end{lemma}

\begin{proof}
By Theorem \ref{theorem-tops}, the graph $\Gamma^s(k,q)$ contains tops and every such top is not a star. 
This implies the existence of linearly independent vectors $\bm{x}',\bm{y}',\bm{z}'\in V$ such that 
$$\langle \bm{x}',\bm{y}'\rangle,\, \langle \bm{x}',\bm{z}'\rangle,\, \langle \bm{y}',\bm{z}'\rangle$$
are $2$-dimensional subcodes of simplex codes and and there is no simplex code containing $\bm{z}'$ and $\bm{x}'+a'\bm{y}'$
for a certain scalar $a'$.
For any vectors $\bm{x},\bm{y}\in V$ spanning a $2$-dimensional subcode of a simplex code
and any non-zero scalar $a$ there is a linear isomorphism between $\langle \bm{x}',\bm{y}'\rangle$ and $\langle \bm{x},\bm{y}\rangle$
sending $\langle \bm{x}'\rangle,\langle \bm{y}'\rangle$ to $\langle \bm{x}\rangle,\langle \bm{y}\rangle$ and $\bm{x}'+a'\bm{y}'$ to $\bm{x}+a\bm{y}$.
This isomorphism can be extended to a monomial linear automorphism $l$ of $V$ (by MacWillams theorem).
The vector $\bm{z}=l(\bm{z}')$ is as required.
\end{proof}

For a subset ${\mathcal X}$ in the point set of a geometry we denote by ${\mathcal  X}^\sim$
the set of all points collinear to each point of ${\mathcal X}$.
If $P$ and $Q$ are collinear points of a polar space, then $\{P,Q\}^{\sim\sim}$ is the line containing  $P,Q$.
Theorem \ref{extension} and Lemma \ref{tex} give the following.

\begin{cor}
If  one of the conditions from Theorem {\rm \ref{theorem-tops}} is satisfied,
then $$\{P,Q\}^{\sim\sim}=\{P,Q\}$$
for any collinear points $P,Q\in {\mathcal S}(k,q)$.
\end{cor}

\begin{proof}
Observe that $\{P,Q\}^{\sim\sim}$ is the intersection of all maximal cliques of the collinearity graph 
containing $P,Q$. 
Since every maximal singular subspace of ${\mathcal S}(k,q)$ is a clique of the collinearity graph,
the second part of Theorem  \ref{extension} shows that $\{P,Q\}^{\sim\sim}$ is a subset of
the line containing $P,Q$.
Then Lemma \ref{tex} implies that $\{P,Q\}^{\sim\sim}$ is precisely $\{P,Q\}$.
\end{proof}

\subsection*{Funding}
No fundings.

\subsection*{Author Contribution}
The authors have equal contribution to the manuscript.

\subsection*{Conflict of Interest}
The authors declare that they have no competing interests.

\subsection*{Data Availability Statement}
No data was used for the research described in the article.

\end{document}